\documentclass{amsart}[11pt]

\usepackage{amssymb}
\usepackage{amsmath}
\usepackage{amsfonts}
\usepackage{yfonts}
\usepackage{amsthm}	
\usepackage{bbm}
\usepackage{dsfont}
\usepackage[shortlabels]{enumitem}
\usepackage{multicol}
\usepackage{ragged2e}
\usepackage{wrapfig}
\usepackage{geometry}
\usepackage{graphicx}
\usepackage{tikz}
\usepackage{hyperref}

\hypersetup{citecolor=red}

\newcommand{\R}{\mathbb{R}} 
\newcommand{\C}{\mathbb{C}} 
\newcommand{\N}{\mathbb{N}} 
\newcommand{\D}{\mathbb{D}} 
\newcommand{\T}{\mathbb{T}} 
\newcommand{\sbs}{\subset} 
\newcommand{\eps}{\varepsilon} 
\newcommand{\vphi}{\varphi}
\newcommand{\vth}{\vartheta}

\newcommand{\no}[1]{\Vert #1 \Vert} 

\newcommand{\Ra}{\Rightarrow}
\newcommand{\ra}{\rightarrow}

\renewcommand{\it}[1]{\textit{#1}}
\renewcommand{\bf}[1]{\textbf{#1}}

\makeatletter
\renewcommand\subsection{\@startsection{subsection}{2}%
\z@{-.5\linespacing\@plus-.7\linespacing}{.5\linespacing}%
{\normalfont\scshape}}
\renewcommand\subsubsection{\@startsection{subsubsection}{3}%
\z@{.5\linespacing\@plus.7\linespacing}{-.5em}%
{\normalfont\scshape}}
\makeatother

\newtheorem{theorem}{Theorem}[section]
\newtheorem{lemma}[theorem]{Lemma}
\newtheorem{remark}[theorem]{Remark}
\newtheorem{corollary}[theorem]{Corollary}
\newtheorem{proposition}[theorem]{Proposition}

\theoremstyle{definition}
\newtheorem{definition}{Definition}

\newtheorem*{manuallemma}{Lemma 13, \cite{Abakumov2020}}

\usepackage[backend=biber,style=numeric,sorting=nty,giveninits=true,isbn=false]{biblatex} 
\nocite{*}

\title[High degree simple partial fractions]{ High degree simple partial fractions in the Bergman space: Approximation and Optimization}
\author{Nikiforos Biehler}
\date{March 2024}

\begin{document}

\begin{abstract}
We consider the class of standard weighted Bergman spaces $A^2_{\alpha}(\D)$ and the set $SF^N(\T)$ of simple partial fractions of degree $N$ with poles on the unit circle. We prove that under certain conditions, the simple partial fractions of order $N$, with $n$ poles on the unit circle attain minimal norm if and only if the points are equidistributed on the unit circle. We show that this is not the case if the conditions we impose are not met, exhibiting a new interesting phenomenon. We find sharp asymptotics for these norms. Additionally we describe the closure of these fractions in the standard weighted Bergman spaces.
\end{abstract}

\maketitle

\section{Introduction \& main results}

Let $g: [0,1] \ra \R_+$ be a continuous function such that $g(0) = 0$. The weighted Bergman-Hilbert space $A^2_{(g)}(\D) = A^2_{(g)}$, associated to the weight $g$, is defined as the set of all functions $f$, holomorphic on the unit disc $\D = \{ z \in \C \, : \, |z| < 1 \} $, that satisfy: 

\begin{equation*}
\no{f}^2_{(g)} := k_g \int_{\D}|f(z)|^2 g(1-|z|^2) \,dA(z) < + \infty \, ,
\end{equation*}
where $dA(z)$ denotes the normalized Lebesgue measure on the unit disc, and

\begin{equation}
k_g = \bigg(\int_{\D}g(1-|z|^2)\,dA(z)\bigg)^{-1} = \bigg(\int_0^1g(r)dr\bigg)^{-1} \, ,
\end{equation}
is the normalization constant. Our research will be focused exclusively on the standard class of weighted Bergman spaces, here denoted by $A^2_{\alpha}$, corresponding to the weight $g(r) = r^{\alpha}$, for $\alpha > -1$. A \textit{simple partial fraction} (or sometimes just simple fraction) will be a complex function of the form:

\begin{gather*}
   f(z) = \sum_{0\leq k < n} \frac{1}{z-a_k} ,
\end{gather*}
where $a_k \in \C$. These functions can be represented in multiple ways, notably, as logarithmic derivatives of polynomials, or alternatively as the Cauchy transforms of sums of Dirac masses placed at the points $a_0, a_1, \ldots, a_{n-1}$. Another interpretation of the function $f$ is that the value $f(z)$ represents the complex conjugate of the electrostatic field at the point $z$, caused by charges at the points $a_0, a_1, \ldots ,a_{n-1}$, given that forces are inversely proportional to the distance. Simple partial fractions have been extensively during the last 80 years. 

 For a recent survey on partial fractions see \cite{Danchenko2011}. In 1949, G. R. MacLane \cite{MacLane1949} began the study of polynomial approximation under constraints on the locations of their zeros. Specifically, a set $E$ is called a polynomial approximation set relative to a domain $G$ if every zero-free holomorphic function $f$ defined on $G$ can be uniformly approximated on compact subsets of $G$ by polynomials whose zeros are restricted to $E$. Such polynomials will be called $E$-polynomials. MacLane demonstrated that for any bounded, simply connected Jordan domain $\Omega \sbs \C$ with a rectifiable boundary, $\partial \Omega$ serves as a polynomial approximation set relative to $\Omega$. Further advancements were made by M. Thompson \cite{Thompson1967}, C. K. Chui \cite{Chui1969}, and Z. Rubinstein and E. B. Saff \cite{Rubinstein1971} extending MacLane’s results, focusing on bounded polynomial approximation in the unit disc. Moreover, J. Korevaar \cite{Korevaar1964} generalized MacLane’s problem, connecting it to approximation by simple partial fractions. In his paper, Korevaar essentially showed that approximation by polynomials with constraints on their zeros is equivalent to approximation by simple partial fractions with constraints on their poles. He did that using the notion of an \textit{asymptotically neutral} family. A family of finite sequences $\zeta_{n,1}, \zeta_{n,2}, \ldots, \zeta_{n,n}$, $n = n_j \ra \infty$, will be called asymptotically neutral relative to a domain $D$ if :
\begin{gather*}
    \sum_{k=1}^n \frac{1}{\zeta_{n,k} - z} \ra 0 \,\, ,\,\, \text{as}\,\, n \ra \infty,
\end{gather*}
uniformly on all compact sets of $D$. Korevaar's important result is the following:

\begin{theorem}{(Korevaar)}
Let $D$ be a bounded simply connected domain of the complex plane and $\Gamma $ a set disjoint from $D$. The following statements are equivalent:
\begin{enumerate}[label= (\roman*)]
    \item $\Gamma$ is a polynomial approximation set relative to $D$,
    \item For every $\zeta \in \Gamma $, the function $z \mapsto \frac{1}{\zeta - z}$ can be approximated by $\Gamma$-polynomials in $D$,
    \item $\Gamma$ contains an asymptotically neutral family relative to $D$,
    \item $\overline{\Gamma}$ divides the plane, and $D$ belongs to a bounded component of $\C \setminus \overline{\Gamma}$.
\end{enumerate}
\end{theorem}

 For a given set $E \sbs \C$, consider the set of all simple fractions with poles on $E$:

 \begin{gather*}
     SF(E) = \bigg\{ \sum_{0\leq k < n} \frac{1}{z-a_k} \, : \,  n \in \N , a_k \in E \bigg\}.
 \end{gather*}

A corollary of Korevaar's theorem is MacLane's previous result, i.e. that the boundary $\partial D$ of a given bounded domain $D$ is a polynomial approximation set relative to that domain. Another consequence of Korevaar's theorem is the following corollary: If $G$ is a simply connected bounded domain in $\C$ and $K$ is a compact subset of $G$ with connected complement, then $SF(\partial G)$ is dense in $A(K)$, the algebra of continuous functions on $K$ and holomorphic in its interior. Several variations of this result were later proven by P. A. Borodin in \cite{Borodin2016}, \cite{Borodin2014}, \cite{Borodin2021}.

Another problem related to approximation by simple partial fractions was proposed by Chui in the 1970's \cite{Chui1971}, concerning the density of $SF(\T)$ in the classical Bergman space  $A^1 = A^1_0(\D) $. In particular, he conjectures the following bound:

\begin{gather*}
\bigg\Vert \sum_{0\leq k < n} \frac{1}{z-a_k} \bigg\Vert_{A^1} \geq \bigg\Vert \sum_{0 \leq k <n} \frac{1}{z-e^{2 \pi i k/n}}\bigg\Vert_{A^1},
\end{gather*}
for any $n \in \N$, and any points $a_k \in \T$. Let $\Psi_n(z) = \sum_{0 \leq k <n} \frac{1}{z-e^{2 \pi i k/n}}$, be the function corresponding to equidistributed points on the unit circle. Since the norms $\Vert \Psi_n \Vert_{A^1}$ are uniformly bounded from below, a positive answer to the conjecture would also imply that $SF(\T)$ is not dense in $A^1$. The question of density in $A^1$ was quickly answered by D. J. Newman in \cite{Newman1972}, where he showed that:

\begin{gather*}
    \bigg\Vert \sum_{0\leq k < n} \frac{1}{z-a_k} \bigg\Vert_{A^1}  \geq \frac{\pi}{18},
\end{gather*}
for all $n \in \N$ and $a_k \in \T$, thus demonstrating that $SF(\T)$ is not dense $A^1$.

In a recent paper (\cite{Abakumov2020}), E. Abakumov, A. Borichev and K. Fedorovskiy returned to the longstanding conjecture of Chui, and proved an alternative version of it in the context of weighted Bergman spaces of square integrable functions. In particular they proved that for weights $g$ defined on $[0,1]$, which are non-decreasing and concave, with $g(0)=0$, the following inequality holds:

\begin{gather}
\bigg\Vert \sum_{0\leq k < n} \frac{1}{z-a_k} \bigg\Vert_{A^2_{(g)}} \geq \bigg\Vert \sum_{0 \leq k <n} \frac{1}{z-e^{2 \pi i k/n}}\bigg\Vert_{A^2_{(g)}},
\end{gather}
for any $n \in \N$, and any points $a_k \in \T$. This class of weights includes the Bergman spaces $A^2_{\alpha}$, for $0 < \alpha \leq 1$. The paper also provides asymptotic results for these norms, explores the density of simple partial fractions in various spaces, and gives a refined version of Thompson’s theorem. Regarding the density of simple partial fractions in the classical weighted Bergman spaces, they obtain the following dichotomy:

\begin{gather*}
   \text{clos}_{A^2_{\alpha}}(SF(\T)) =
   \begin{cases}
SF(\T)  \,\, ,& \text{if} \,\,\,\, 0 < \alpha \leq 1, \\
A^2_{\alpha}  \,\, ,& \text{if} \,\,\,\, \alpha > 1. 
\end{cases}
\end{gather*}

The focus of this paper will be shifted to a similar class of functions, namely \textit{simple partial fractions of degree $N$}, that is, functions of the form:

\begin{gather*}
   f(z) = \sum_{0 \leq k < n} \frac{1}{(z-a_k)^N}, \, a_k \in \C,
\end{gather*}
for $N > 1$. They are derivatives of the simple partial fractions described above. 

Derivatives of simple partial fractions have been recently studied by N. A. Dyuzhina in \cite{Dyuzhina2021}, in the context of Hardy spaces on the half-plane $\Pi_+$. Results from \cite{Danchenko1995} imply that simple fractions with poles on the lower half-plane are not dense in the space $H^p(\Pi_+)$. Dyuzhina proved that if one instead considers derivatives of simple partial fractions with poles on the lower half-plane, then a density theorem holds for all $1 <p<\infty$ for the Hardy space $H^p(\Pi_+)$. Furthermore she shows the density of the same set on the subspace of $H^p(\Pi_+)$ consisting of functions which can be extended continuously to the entire real line.

As a first result, we establish the connection between Korevaar's theorem (Theorem 1.1) and high degree simple fractions, and prove that simple partial fractions (of degree 1) are not particularly special in the context of Korevaar's theorem. Specifically, his result can be rephrased in terms of simple partial fractions of degree $N$, for any $N\geq 1$.

\begin{theorem}
     Let $D \sbs \C$ be a bounded simply connected domain,  $\Gamma$ a bounded disjoint set from $D$, and $N \in \N$. The following are equivalent:
\begin{enumerate}[label= (\roman*)]
    \item $\Gamma$ is a polynomial approximation set relative to $D$,
    \item For every $\zeta \in \Gamma$, the function $z \mapsto \frac{1}{(\zeta -z)^N}$ can be approximated by $\Gamma$-polynomials in $D$,
    \item $\Gamma$ contains an N-asymptotically neutral family relative to $D$,
    \item $\overline{\Gamma}$ divides the plane and $D$ lies in a bounded connected component of $\C \backslash \overline{\Gamma}$.
\end{enumerate}     
\end{theorem}

Here, \textit{$N$-asymptotically neutral}, means replacing simple partial fractions, by simple partial fractions of degree $N$ in the definition of asymptotically neutral. As a corollary we get:

\begin{corollary}
 For any simply connected region $D \sbs \C$, the set $SF^N(\partial D)$ is dense in $Hol(D)$, the set of holomorphic functions in $D$, with the topology of uniform convergence on compact subsets of $D$.
\end{corollary}

Next we discuss the problem of placing $n$ points $a_0, a_1, \ldots, a_{n-1}$ on the unit circle to minimize the Bergman norm of the corresponding simple partial fraction of degree $N$. The set $SF^N(\T)$, of simple partial fractions of degree $N$, is contained in the Bergman space $A^2_{\alpha}$ if and only if $\alpha > 2(N-1)$. A natural conjecture would be to have an equivalent result as in (1), namely:

\begin{gather}
    \min_{a_k \in \T} \bigg \Vert \sum_{0 \leq k < n}\frac{1}{(z-a_k)^N} \bigg \Vert_{\alpha} = \no{\Psi_n^N}_{\alpha},
\end{gather}
where $\Psi_n^N(z) = \sum_{0 \leq k <n} \frac{1}{(z-e^{2 \pi i k/n})^N}$. Interestingly, this ceases to be true in general for $N>1$, as will be demonstrated through some analytical as well with some numerical calculations in section 4. It is a new phenomenon, and completely unexpected, that the optimal placement for two points, in the case of $N=2$ and $\alpha = 3$ is not on opposite sides of the diameter of the unit circle, but rather it forms an acute angle which is numerically calculated.

In order to compensate for this irregularity that arises once the degree of the functions is higher than one, we attempt to impose some extra conditions on the moments of the points placed on the unit circle. In particular, we show that if we restrict ourselves to points whose $N^2-2$ first moments are assumed to be equal to zero, then the optimal placement for them is, up to a rotation, a regular $n$-gon. 

\begin{theorem}
    Let $N > 1$ and $\alpha_{\ast} = 2(N-1) + 1$. For each $n\in \N$, consider the set $W_n$ defined as:

    \begin{gather*}
        W_n := \bigg\{ (a_0, a_1, \ldots, a_{n-1}) \in \T^n \, : \, a_i \neq a_j \,\, \text{for} \,\,i \neq j\, , \sum_{0 \leq k < n} a_k^m =0 , m = 1,2,\ldots , N^2-2 \bigg\}.
    \end{gather*}

Then:

\begin{gather*}
\min_{(a_0, a_1, \ldots, a_{n-1}) \in W_n} \bigg \Vert \sum_{0 \leq k < n} \frac{1}{(z-a_k)^N} \bigg \Vert_{\alpha_{\ast}} = \no{\Psi_n^N}_{\alpha_{\ast}}.
\end{gather*}
    
\end{theorem}

Subsequently we provide some asymptotic estimates for the minimum involved in (2) as well as for the sequence of norms $\no{\Psi_n^N}_{\alpha}$. As per usual $\Gamma$ will be the Gamma function and $\zeta$ the Riemann Zeta function.

\begin{theorem}
    For every $\alpha > 2(N-1)$ we have:

    \begin{gather*}
    \lim_{n \ra \infty} n^{\alpha + 1- 2N  }  \Vert \Psi_n^N \Vert_{\alpha}^2 = \frac{\Gamma(\alpha +2) \zeta(\alpha + 1 - 2(N-1))}{((N-1)!)^2}.
    \end{gather*}
    
\end{theorem}

Theorem 1.5 reveals that even though the optimal placement of points is not equidistribution in general, it cannot diverge too much from it. In fact, relation (2) is true up to a constant (independent of $n$). We will use the following notation throughout the paper: For positive quantities $A$ and $B$, $A \lesssim B$ will mean that  there exists some positive constant $C$ such that $A \leq C B $. Moreover we will use $A \asymp B$ to mean both $A \lesssim B $ and $B \lesssim A$.

\begin{theorem}
    Let $\alpha > 2(N-1)$. Then:

    \begin{gather*}
        \min_{a_k \in \T} \bigg \Vert \sum_{k=0}^{n-1} \frac{1}{(z-a_k)^N} \bigg \Vert_{\alpha} \asymp \Vert \Psi_n^N \Vert_{\alpha}.
    \end{gather*}
\end{theorem}

Finally, we study the density of $SF^N(\T)$ in the Bergman space $A^2_{\alpha}$, as a function of the parameter $\alpha$. We obtain the following result:

\begin{theorem}
Let $\alpha_{\ast} = 2(N-1)+1$, and $N \geq 2$. Then the set of partial fractions of degree $N$, $SF^N$ is nowhere dense in $A^2_{\alpha}$ when $2(N-1)< \alpha < \alpha_{\ast}$, and it is simply not dense when $\alpha = \alpha_{\ast}$. It is dense in $A^2_{\alpha}$ when $\alpha > \alpha_{\ast} + 1$.
\end{theorem}

The last theorem rely essentially on Borodin's work in \cite{Borodin2014} and the estimates established in Theorem 1.5.

The text is organized as follows. Section 2 is devoted the proof of Theorem 1.2. Sections 3 and 4 are devoted to some preliminary results that will be used in the proofs of the main theorems. In section 5 we show that equidistribution is not optimal for $N>1$. Section 6 contains the proof of Theorem 1.4. Theorems 1.5 and 1.6 are the subject of section 7. The questions concerning density are finally discussed in section 8.


\section{Korevaar's theorem for high order simple partial fractions}

The theorem of Korevaar establishes in precise terms the link between polynomial approximation with constraints on the zeros and approximation by simple partial fractions with constraints on the poles. In this short section we show that the role of simple partial fractions $SF= SF^1$ is not unique here, but may be substituted by $SF^N$. We will say that a system of finite families $\{ \zeta_{n,k} \, : \, 0 \leq k < n \} \sbs E$ is N-asymptotically neutral, relative to $G$ if 

\begin{gather*}
\sum_{0 \leq k < n} \frac{1}{(\zeta_{n,k}-z)^N} \ra 0,
\end{gather*}
on all compact sets of $G$, as $n \ra \infty$. We may now prove Theorem 1.2.

\vspace{2mm}

\textit{Proof of Theorem 1.2:}

By Korevaar's theorem, it suffices to show the implications $(i) \Ra (ii) \Ra (iii) \Ra (iv)$. The implication $(i) \Ra (ii) $ is part of Korevaar's theorem, since $z \mapsto \frac{1}{(\zeta - z)^N}$ is non-vanishing in $D$ for every $\zeta \in \T$. For $(ii) \Ra (iii) $, fix some $\zeta \in \T$. Since $\frac{1}{(\zeta - z)^N}$ can be approximated by $\Gamma$-polynomials, if we consider such a sequence of polynomials $p_n$ converging to $\frac{1}{(\zeta - z)^N}$, then the degree $\deg p_n$ must tend to infinity, otherwise an appropriate subsequence will converge to a polynomial. Taking logarithmic derivatives we obtain:

\begin{gather*}
    \sum_{0 \leq k < N_n} \frac{1}{\zeta_{N_n,k}-z} \ra -\frac{N}{\zeta - z},
\end{gather*}
where $\zeta_{N_n,k}$, $0 \leq k < N_n$ are the zeros of the polynomials $p_n$, and convergence is uniform on every compact subset of $D$. Differentiating $N-1$ times yields:

\begin{gather*}
    \sum_{0 \leq k < N_n} \frac{(-1)^{N-1}(N-1)!}{(\zeta_{N_n,k}-z)^N} \ra -\frac{(-1)^{N-1}N!}{(\zeta - z)^N}.
\end{gather*}

By simplifying and rearranging terms we get that the family $\zeta_{N_n,k}, \zeta, \ldots, \zeta$, $0 \leq k < N_n$ forms and $N$-asymptotically neutral family in $\Gamma$ relative to $D$. It remains to show implication $(iii) \Rightarrow (iv)$. This is essentially Lemma 2.1 of \cite{Korevaar1964}, which we briefly explain for completeness. Let $U$ be the union of the points of the $N$-asymptotically neutral family in $\Gamma$. Then $\overline{U}$ divides the plane, and $D$ belongs to the bounded connected component of $\C \setminus \overline{U}$. Indeed, let $D'$ be the component of the complement of $\overline{U}$ which contains $D$ and suppose it is unbounded. Consider the functions:

\begin{gather*}
    f_n(z) = \frac{1}{N_n} \sum_{0 \leq k < N_n} \frac{1}{(\zeta_{N_n,k}-z)^N}, \,\, z \in D',
\end{gather*}
where $\zeta_{N_n,k}$, $0 \leq k < N_n$ designates the $N$-asymptotically neutral family in $\Gamma$. The functions $f_n$ are holomorphic and locally uniformly bounded on all compact subsets of $D'$. Moreover, by the hypothesis, $f_n \ra 0$ on all compact subsets of $D$. By Vitali's theorem, an appropriate subsequence of those $f_n$ has to converge to zero on all of $D'$. This is a contradiction, since $D'$ is supposed to be unbounded which means that far away from $\overline{U}$ all of the functions $f_n$ must be bounded below by a positive constant. This concludes the proof.
$\hfill \blacksquare$

An almost immediate consequence is Corollary 1.3.

\vspace{2mm}

\textit{Proof of Corollary 1.3:}

Given $f \in Hol(D)$, it suffices to take a function $F$ such that $F^{(N)}(z) = (N-1)! f(z)$ and consider $G(z) = \exp(F(z))$. Then a sequence $\{p_n\}$ of $\partial D$-polynomials converges to $G$ on compact sets of $D$. By taking the logarithm and differentiating $N$ times we get:

\begin{gather*}
    \sum_{0 \leq k < N_n} \frac{1}{(\zeta_{N_n,k}-z)^N} \ra f(z),
\end{gather*}
on all compact sets of $D$.

$\hfill \blacksquare$ 

\section{Preliminary results}

In this section we present some calculations which will be used further in the text. First, we consider a general weight $g$ as mentioned in the introduction, and derive the condition it must
satisfy in order for the simple partial fractions of degree $N$ to belong to the Bergman space associated to that weight. Next, we consider the minimization problem that we aim to solve and introduce
an auxiliary function which arises naturally in its study. We provide multiple formulas for that function, each of which has its own advantages and disadvantages.

\subsection{Condition on weight}

We will determine the condition the weight $g$ has to satisfy in order for us to have $SF_N \subset A^2_{(g)}$. It
is clear that $SF_N \subset A^2_{(g)}$ if and only if $\frac{1}{(z-1)^N} \in A^2_{(g)}$. We start with a simple lemma:

\begin{lemma}
Let \( p(r) \) be a polynomial such that \( p(1) \neq 0 \). Then for every \( n \in \mathbb{N} \) we have that:
\[
\frac{d^n}{dr^n} \left( \frac{p(r)}{(r - 1)^N} \right) \asymp \frac{1}{(r - 1)^{N+n}} , \quad r \ra 1.
\]
\end{lemma}

\begin{proof}
We proceed by induction on \( n \). For \( n = 1 \) we have:
\[
\frac{d}{dr} \left( \frac{p(r)}{(r - 1)^N} \right) =
\frac{p'(r)(r - 1)^N - p(r)N(r - 1)^{N-1}}{(r - 1)^{2N}}
= \frac{p'(r)(r - 1) - Np(r)}{(r - 1)^{N+1}}.
\]
Since \( p(1) \neq 0 \), the polynomial \( p'(r)(r - 1) - Np(r) \) has the property
\( p'(r)(r - 1) - Np(r) \big|_{r=1} \neq 0 \). We continue in that manner and at the step \( n \) we obtain a polynomial \( q \) such that \( q(1) \neq 0 \) such that
\[
\frac{d^n}{dr^n} \left( \frac{p(r)}{(r - 1)^N} \right) =
\frac{q(r)}{(r - 1)^{N+n}}.
\]
We differentiate once more to get:
\[
\frac{d^{n+1}}{dr^{n+1}} \left( \frac{q(r)}{(r - 1)^{N+n}} \right) =
\frac{q'(r)(r - 1) - (N + n)q(r)}{(r - 1)^{N+n+1}}.
\]
The polynomial on the numerator is non-zero in a neighbourhood around \( r = 1 \), so we get:
\[
\frac{d^{n+1}}{dr^{n+1}} \left( \frac{q(r)}{(r - 1)^{N+n}} \right) \asymp \frac{1}{(r - 1)^{N+n+1}}, \quad r \ra 1.
\]
\end{proof}

We will use the following notation to simplify large formulas. If $x \in \mathbb{R}$ and $n \in \mathbb{N}$ then
$(x)_n = x(x - 1) \cdots (x - n + 1)$ will represent the falling factorial. We write $\langle \cdot , \cdot \rangle_{(g)}$ for the inner product in $A^2_{(g)}$.

\begin{lemma}
Let $N \in \N$ and $\vth, \vphi \in [0,2\pi]$. Then:

\begin{gather*}
\text{Re} \bigg\langle \frac{1}{(z-e^{i \vth})^N},\frac{1}{(z-e^{i \vphi})^N} \bigg \rangle_{(g)} = \frac{k_g}{((N-1)!)^2} \sum_{n=N}^{\infty} c_n(g,N) \cos(n(\vphi - \vth)),
\end{gather*}
where:

\begin{gather}
c_n(g,N) = ((n-1)_{N-1})^2\int_0^1 r^{n-N}g(1-r)dr.
\end{gather}
and $k_g$ is defined in (1).
\end{lemma}

\begin{proof}
For every $z \in \D$ and $t \in [0,2\pi]$ we have:

\[ 
\frac{1}{z-e^{it}} = -e^{-it} \sum_{n=0}^{\infty}(ze^{-it})^n.
\]

Differentiating the two expressions $N-1$ times with respect to $z$ and rearranging terms we get:

\begin{gather}
\frac{1}{(z-e^{it})^N} = \frac{(-1)^N}{(N-1)!} \sum_{n=N-1}^{\infty}(n)_{N-1}z^{n-N+1}e^{-i(n+1)t} = \frac{(-1)^N}{(N-1)!}\sum_{n=N}^{\infty} (n-1)_{N-1} z^{n-N}e^{-int}.
\end{gather}

Moreover, a standard calculation shows that:

\begin{gather*}
\int_{\D} z^k \overline{z}^l g(1-|z|^2)\,dA(z) = 
\begin{cases}
0, \,\,k \neq l, \\
\int_0^1 r^k g(1-r)\,dr , \,\,k = l.
\end{cases}
\end{gather*}

The result follows by combining the above two formulas.
\end{proof}

\begin{proposition}
We have \( SF_N \subset A^2_{(g)} \) if and only if
\[
\int_{0}^{1} \frac{g(r)}{r^{2N-1}} dr < \infty.
\]
In particular, if \( g(r) = r^\alpha \) then
\[
SF_N \subset A^2_{\alpha} \iff \alpha > 2(N - 1).
\]
\end{proposition}

\begin{proof}
Using Lemma 3.2 we may write
\[
\left\| \frac{1}{(z - 1)^N} \right\|^2_{(g)} =
\left\langle \frac{1}{(z - 1)^N}, \frac{1}{(z - 1)^N} \right\rangle_{(g)} = \frac{k_g}{((N-1)!)^2} \sum_{n=N}^{\infty}c_n(g,N) .
\]

Using (4) we see that:

\begin{gather*}
    \frac{(-1)^N}{(N-1)!}\sum_{n = N}^{\infty}(n-1)^2_{N-1}r^{n-N} = \frac{d^{N-1}}{dr^{N-1}}\bigg( \frac{r^{N-1}}{(r-1)^N} \bigg).
\end{gather*}

As a result we may re-sum the series:

\begin{gather*}
    \left\| \frac{1}{(z - 1)^N} \right\|^2_{(g)} = \frac{(-1)^N k_g}{(N-1)!} \int_0^1 \frac{d^{N-1}}{dr^{N-1}}\bigg( \frac{r^{N-1}}{(r-1)^N} \bigg) g(1-r) \,dr.
\end{gather*}

The above integral is finite if and only it is finite in a neighbourhood of $r=1$. Hence, according to Lemma 3.1:

\begin{gather*}
    \int_0^1 \frac{d^{N-1}}{dr^{N-1}}\bigg( \frac{r^{N-1}}{(r-1)^N} \bigg) g(1-r) \,dr \asymp \int^1 \frac{g(1-r)}{(r-1)^{2N-1}}\,dr \asymp \int_0 \frac{g(r)}{r^{2N-1}}\,dr.
\end{gather*}

This is the condition for a general weight. If furthermore $g(r) = r^{\alpha}$ we see that this reduces to \( \alpha > 2(N - 1) \), completing the proof.
\end{proof}

\subsection{Interaction function and alternative forms}

Consider a simple partial fraction of degree \( N \),

\[
f(z) = \sum_{k=0}^{n-1} \frac{1}{(z - e^{i\vth_k})^N},
\]
with \( \vth_0 \leq \vth_1 \leq \dots \leq \vth_{n-1} \in [0, 2\pi) \). The problem we are interested in is to try to calculate, for a figen $n \in \N$ the configurations of \( \vth_0 \leq \vth_1 \leq \dots \leq \vth_{n-1}\) for which the following minimum is attained:

\[
\min_{\vth_0 \leq \vth_1 \leq \dots \leq \vth_{n-1} \in [0,2\pi)} \| f \|_{(g)}.
\]

A simple calculation yields:

\begin{gather*}
\| f \|^2_{(g)} =
\left\langle \sum_{k=0}^{n-1} \frac{1}{(z - e^{i\vth_k})^N}, \sum_{k=0}^{n-1} \frac{1}{(z - e^{i\vth_k})^N} \right\rangle_{(g)} \\ = \sum_{k=0}^{n-1} \left\| \frac{1}{(z - e^{i\vth_k})^N} \right\|^2_{(g)}
+ \sum_{k \neq j} \left\langle \frac{1}{(z - e^{i\vth_k})^N}, \frac{1}{(z - e^{i\vth_j})^N} \right\rangle_{(g)} \\
= n \left\| \frac{1}{(z-1)^N}\right\|^2_{(g)} + \sum_{k \neq j} \text{Re}\left\langle \frac{1}{(z - e^{i\vth_k})^N}, \frac{1}{(z - e^{i\vth_j})^N} \right\rangle_{(g)}.
\end{gather*}

In light of Lemma 3.2 we define the \textit{interaction function}:

\begin{gather}
\phi_{(g),N} (\vth) = \frac{k_g}{((N-1)!)^2} \sum_{n=N}^{\infty} c_n(g,N)\cos(n\vth).
\end{gather}

Thus:

\[
\text{Re}\left\langle \frac{1}{(z - e^{i\vth_k})^N}, \frac{1}{(z - e^{i\vth_j})^N} \right\rangle_{(g)} = \phi_{(g),N} (\vth_j - \vth_k),
\]
which permits us to restate the initial minimization problem as the following, equivalent, problem:

\[
\min_{\vth_0 \leq \vth_1 \leq \dots \leq \vth_{n-1} \in [0,2\pi)}
\sum_{k \neq j} \phi_{(g),N} (\vth_j - \vth_k).
\]

To obtain a closed form for the interaction function, we proceed as before. Using the integral representation of the inner product, and developing the series we arrive at:

\begin{gather*}
\phi_{(g),N} (\vth) =
\frac{(-1)^N k_g}{ (N-1)!} \text{Re}\int_{0}^{1} \frac{d^{N-1}}{dr^{N-1}} \left( \frac{r^{N-1}}{(r - e^{-i\vth})^N} \right) g(1 - r) dr.
\end{gather*}

Assuming that the weight satisfies:

\begin{gather}
g^{(k)}(0) = 0,\,\,  k = 0, 1, \dots, N-2,
\end{gather}
we may  integrate by parts \( N-1 \) times and obtain:

\begin{gather*}
\phi_{(g),N} (\vth) =
\frac{(-1)^N k_g}{ (N-1)!} \text{Re}\int_{0}^{1} \frac{r^{N-1}}{(r - e^{-i\vth})^N} g^{(N-1)}(1 - r) dr.
\end{gather*}

A simple computation shows:

\begin{gather}
\text{Re} \left( \frac{(r e^{i\vth})^N}{(r e^{i\vth} - 1)^N} \right)
= \frac{r^{2N}}{(1 + r^2 - 2r \cos\vth)^N} \sum_{k=0}^{N} \binom{N}{k} (-1)^k r^{-k} \cos(k\vth),
\end{gather}
and hence, under condition (7), we have:

\begin{gather}
\phi_{(g),N} (\vth) =
\frac{k_g}{(N-1)!} \int_{0}^{1} \frac{ \sum_{k=0}^{N} \binom{N}{k} (-1)^{N-k} r^{N-k} \cos(k\vth)}{(1 + r^2 - 2r \cos\vth)^N} r^{N-1} g^{(N-1)}(1 - r) dr.
\end{gather}

\begin{lemma}
Let 
\[
U_N(r,\vartheta) = \frac{\sum_{k=0}^{N} \binom{N}{k} (-1)^{N-k} r^{N-k} \cos(k\vartheta)}{(1 + r^2 - 2r \cos(\vartheta))^N}.
\]
Then:
\begin{gather}
\frac{d}{dr} U_N(r,\vartheta) = N \cdot U_{N+1}(r,\vartheta).
\end{gather}
\end{lemma}

\begin{proof}
By (8) we have that: 

\begin{gather*}
    (-r)^N U_N(r, \vth) = \text{Re}\left( \frac{(r e^{i\vth})^N}{(r e^{i\vth} - 1)^N} \right)
\end{gather*}

Taking the derivative we arrive at (10).
\end{proof}

Using Lemma 3.4 we may do integration by parts in (9) to get:
\begin{gather}
\varphi_{(g),N}(\vartheta) = \frac{k_g (-1)^N}{((N-1)!)^2} \int_0^1 \frac{r - \cos(\vartheta)}{1 + r^2 - 2r \cos(\vartheta)} \left( r^{N-1} g^{(N-1)}(1 - r) \right)^{(N-1)} \, dr
\end{gather}
\[
= \frac{k_g (-1)^{N+1}}{((N-1)!)^2} \int_0^1 \frac{1}{2}\log(1 + r^2 - 2r \cos(\vartheta)) \left( r^{N-1} g^{(N-1)}(1 - r) \right)^{(N)} \, dr.
\]
under the condition that $g^{(k)}(0) = 0$ up to $k = 2(N-1)$, which is true for the weights in the scope of this paper, namely $g(r) = r^\alpha$ with $\alpha > 2(N-1)$. 

\section{Modification of asymptotically convex sequences}

This section is dedicated to the proof of Proposition 4.3, which will be used in the proof of Theorem 1.4. Proposition 4.3 is roughly a precise formulation of the following idea: A sequence that is asymptotically convex (convex after a certain point), can be modified in finitely many places to get a convex sequence. We present a proof of this result since it is tailor-made for the particular sequences appearing later in the text.

\begin{definition}
A sequence $\{a_n\}_{n \geq 0} \sbs \R$ is called convex if $2a_n \leq a_{n-1} + a_{n+1}$ for $n \geq 1$. It will be called convex at $n_0 \geq 1 \in \N$ if $2a_{n_0} \leq a_{n_0-1} + a_{n_0+1}$. When we talk about a strictly convex sequence we mean a convex sequence that satisfies the corresponding strict inequality. We will call the sequence $\{a_n\}_{n \geq 0}$ asymptotically (strictly) convex, if it is (strictly) convex after some index $n_0 \in \N$.
\end{definition}

Writing $\Delta a_n = a_n - a_{n+1}$ and $\Delta^2 a_n = \Delta a_n - \Delta a_{n+1}$ the condition of convexity may be written as $\Delta^2 a_n \geq 0$. A positive convex sequence that is additionally bounded must be decreasing and thus of finite limit, and moreover satisfies $n \Delta a_n \ra 0$ as $n \ra + \infty$. The reason we are discussing convex sequences is because they will be used in conjunction with Bari's Theorem \cite{Bary2014}, Chapter 1, Section 30.

\begin{theorem}{(Bari's Theorem)}
  Let $\{a_n\}_{n\geq 0} \sbs \R$ be a positive, decreasing, and convex sequence, of limit zero. Then:

  \begin{gather*}
      \frac{a_0}{2} + \sum_{n=0}^{\infty} a_n \cos(n \vth) \geq 0 \,\, , \quad \text{for all} \,\, \vth \in (0, 2\pi).
  \end{gather*}
\end{theorem}

The proof of the theorem relies on the fact that the series can be rewritten as :

\begin{gather*}
    \frac{a_0}{2} + \sum_{n=0}^{\infty} a_n \cos(n \vth) = \frac{1}{2}\sum_{n=0}^{\infty} (n+1) \Delta^2 a_n F_{n+1}(\vth) \,\, , \quad \text{for all} \,\, \vth \in (0, 2\pi), 
\end{gather*}

where $F_n(\vth) = \frac{1}{n} \bigg( \frac{\sin(n\vth/ 2)}{\sin(\vth /2)}\bigg)^2 \geq 0$ are the (positive) Féjer kernels.

\begin{remark}
    From the alternative writing of the cosine series in Bari's theorem, we may deduce that it is strictly positive whenever the sequence $\{a_n\}_{n\geq 0}$ is strictly convex. If the sequence is instead just convex, and in particular if $\Delta^2 a_0 = 0$ then the series might have zeros. The zeros of each Féjer kernel are countable, and as a result the zeros of the whole series are countable. Therefore a series like that cannot vanish on a set of positive measure, particularly it cannot vanish on any interval.
\end{remark}

\begin{proposition}
     Let $\{a_n\}_{n \geq 0} \sbs \R$ be a positive, decreasing sequence with limit zero. If $\{a_n\}_{n \geq 0}$ is asymptotically convex, then there exists some $N_0 \in \N$ and a sequence $\{\tilde{a}_n\}_{n\geq 0}$ with $\tilde{a}_0 = a_0$ and $\tilde{a}_{N_0+k} = a_{N_0+k}$ for every $k \geq 0$ such that:

    \begin{enumerate}[label = (\arabic*)]
        \item $\tilde{a}_n \geq 0$ for every $n \geq 0$,
        \item $\{\tilde{a}_n\}_{n \geq 0}$ strictly decreasing,
        \item $ \{ \tilde{a}_n \}_{n \geq 0}$ convex.
    \end{enumerate}
\end{proposition}

\begin{proof}

We define:

\begin{gather*}
N_0 = \min\{ N \in \N \, : \, a_0 \geq a_N + N\Delta a_N \,\,\text{and}\,\, a_n + n \Delta a_n \,\, \text{decreasing for} \,\, n \geq N \}. 
\end{gather*}

Since $\{a_n\}$ is decreasing with limit zero, and $\{a_n\}$ asymptotically convex, it follows that $a_n + n \Delta a_n$ is eventually decreasing ant of limit zero. This justifies the existence of $N_0$. We will define $\tilde{a}_{N_0-1}$ and the rest of the terms will be defined in the same manner recursively. Consider a real number $\eps$ such that

    \begin{gather*}
 0 \leq \eps \leq \frac{a_0 - ((N_0+1)a_{N_0} - N_0a_{N_0+1})}{N_0},
    \end{gather*}
and set $\tilde{a}_{N_0-1} = 2a_{N_0} - a_{N_0+1}+ \eps$. We verify that 

\begin{enumerate}[label=(\alph*)]
\item  $\tilde{a}_{N_0-1} \geq a_{N_0}$ ,
\item $2a_{N_0} \leq \tilde{a}_{N_0-1} + a_{N_0 +1}$ ,
\item $a_0 \geq N_0 \tilde{a}_{N_0-1} - (N_0-1)a_{N_0}$.
\end{enumerate}

Note that these conditions indeed permit to recursively define the previous term in the same manner. For condition (a), we see that $\tilde{a}_{N_0-1} - a_{N_0} = a_{N_0} - a_{N_0+1} + \eps \geq 0$ since $\eps \geq 0 $ and $a_n$ decreasing. For (b), we have $\tilde{a}_{N_0-1} + a_{N_0 +1} = 2a_{N_0 }+ \eps \geq 2a_{N_0}$. Finally for (c), we write $a_0 - ( N_0 \tilde{a}_{N_0-1} - (N_0-1)a_{N_0})  =a_0 - (N_0 + 1)a_{N_0} + N_0 a_{N_0 +1} - N_0 \eps \geq 0$ according to the condition applied to $\eps$. With $\tilde{a}_{N_0-1}$ defined we may proceed the recursion.
\end{proof}


\section{Counterexamples to equidistribution}

For a given degree $N \in \N$ and a certain number of points $n \in \N$ we consider the optimization problem:

\begin{gather}
\min_{\theta_0 < \theta_1 < \dots < \theta_{n-1} \in [0,2\pi)} \| f \|_{\alpha},
\end{gather}
where:

\[
f(z) = \sum_{k=0}^{n-1} \frac{1}{(z - e^{i \vth_k})^N},
\]
and $\alpha > 2(N-1)$, the parameter corresponding to the weight of the ambient space $A^2_{\alpha}$. We are interested in finding the optimal placement of the numbers $e^{i \vth_k} \in \T$ such that the value in (12) is attained. As already seen in section 3, this problem is essentially the same as minimizing the function:

\begin{gather}
\min_{\vth_0 < \vth_1 < \cdots < \vth_{n-1} \in [0,2\pi)}\sum_{k \neq j}  \vphi_{(r^\alpha),N}(\vth_j - \vth_k),
\end{gather}
where $\vphi_{(r^{\alpha}),N} := \vphi_{\alpha,N}$, is the interaction function as defined in section 3. Note that according to (1) we have that $k_{r^{\alpha}} := k_{\alpha} = \alpha + 1$. We will show that in certain cases, this function is not minimized when the points $e^{i \vth_k} \in \T$ are equidistributed on the unit circle. In other words, the minimum is not attained when $\vth_k = 2k \pi /n$, $k=0,1,\ldots, n-1$.

\subsection{Optimal placement for 2 points}

We consider the case of $N = 2$ and $\alpha = 3$. According to (4), (6) and (11), the interaction function $\vphi_{3,2} = \vphi$ is given by the expressions:

\begin{align*}
  \vphi(\vth) &= 4!\bigg( \sum_{n=1}^{\infty} \frac{n-1}{n(n+1)(n+2)} \cos(n \vth)  \bigg) \\
   & = 4 \int_0^1 \frac{r - \cos(\vth)}{1 + r^2 - 2r\cos(\vth)} (3 - 12r + 9r^2)\,dr.
\end{align*}

We consider $\vth_0 < \vth_1 \in [0,2\pi)$, and minimize the quantity in (5), which in our case is:

\begin{gather*}
    \vphi(\vth_0 - \vth_1) + \vphi(\vth_1 - \vth_0).
\end{gather*}

 The function $\vphi$ is even, and satisfies $\vphi(\pi - \vth) = \vphi(\pi + \vth)$. Moreover the problem is rotation invariant. Therefore we may assume that $\vth_0 = 0$ and $\vth_1 := \vth \in (0,\pi]$. The quantity to minimize thus becomes simply equal to $2\vphi(\vth)$. Without loss of generality, we will try to minimize the function $\vphi(\vth)$ in $[0,\pi]$. First of all we calculate some values for $\vphi$. Using the second expression  we can easily see that:

\begin{gather*}
\vphi(0) = 4 \int_0^1 \frac{r-1}{1 + r^2 - 2r} 3(3r-1)(r-1) \,dr = 12 \int_0^1 (3r-1) \,dr  = 6.
\end{gather*}

Additionally, we may verify using the same expression that:

\begin{gather*}
\vphi(\pi/2) = 4 \int_0^1 \frac{r}{1+r^2}(3 - 12r + 9r^2) \,dr  = -30 + 12 \pi - 12\log(2) \approx -0,618,
\end{gather*}

\begin{gather*}
\vphi(\pi/3) = 4 \int_0^1 \frac{r- 1/2}{1+r^2-r}(3 - 12r + 9r^2) \,dr  = -12 + 2\pi\sqrt{3} \approx -1,117,
\end{gather*}
and 
\begin{gather*}
\vphi(\pi) = 4 \int_0^1 \frac{r+1}{1 + r^2 + 2r} 3(r-1)(3r-1) \,dr = 96\log2 - 66 \approx 0,542.
\end{gather*}

Looking at these values we deduce that this $\vphi$ is not a convex function on $(0,2\pi)$, and hence might not have a single minimum. Moreover $\vphi(\pi/2) < \vphi(\pi)$, so the minimum is not achieved at $\vphi(\pi)$ which would correspond to equidistributed points. The reason for calculating $\vphi(\pi/3)$ will become apparent in section 6. Another application of (11) gives:

\begin{gather*}
    \vphi(\vth) = 12 \int_0^1 \log(1 + r^2 - 2r \cos(\vth)) (2 - 3r)\,dr .
\end{gather*}

Differentiating the above integral gives:

 \begin{gather*}
     \vphi'(\vth) = 24 \sin(\vth) \int_0^1 \frac{2r-3r^2}{1+r^2-2r\cos(\vth)}\,dr.
 \end{gather*}

The above integral can be evaluated as:

\begin{gather*}
\int_0^1 \frac{2r-3r^2}{1+r^2-2r\cos(\vth)}\,dr \\
= -3 + (1 -3 \cos(\vth))\log(2-2\cos(\vth)) + \frac{2\cos(\vth)(1-3\cos(\vth))+3}{\sin(\vth)}\cdot \arctan\bigg(\frac{1 + \cos(\vth)}{\sin(\vth)}\bigg) .
\end{gather*}

And hence the minimum can be obtained as a solution of the following equation in $(0, \pi)$:

\begin{align*}
-3\sin(\vth) +  \sin(\vth)(1 -3 \cos(\vth))\log(2-2\cos(\vth)) \\ +\sin(\vth)(2\cos(\vth)(1-3\cos(\vth))+3)\arctan\bigg(\frac{1 + \cos(\vth)}{\sin(\vth)}\bigg) = 0.
\end{align*}

A standard numerical approximation using Brent’s method with hyperbolic extrapolation, reveals that $\vth_{min} \approx 0.91981$. 

An alternative approach to calculate the minimum is by approximating the interaction function by its cosine series. That way we produce its graph and calculate the point at which the minimum is attained.
The approximation implemented, used the first 3000 terms of the cosine series. To find the minimum a brute force algorithm, was used which sampled 10.000 points in the interval $[0,\pi]$. This numerical calculation reveals  that the minimum is attained approximately at $\vth_{min} \approx 0.91932 \approx \frac{12\pi}{41}$ with $\vphi(\vth_{min}) \approx -1.14963$:

\begin{figure}[htbp]
\centering
\includegraphics[width=0.8\textwidth]{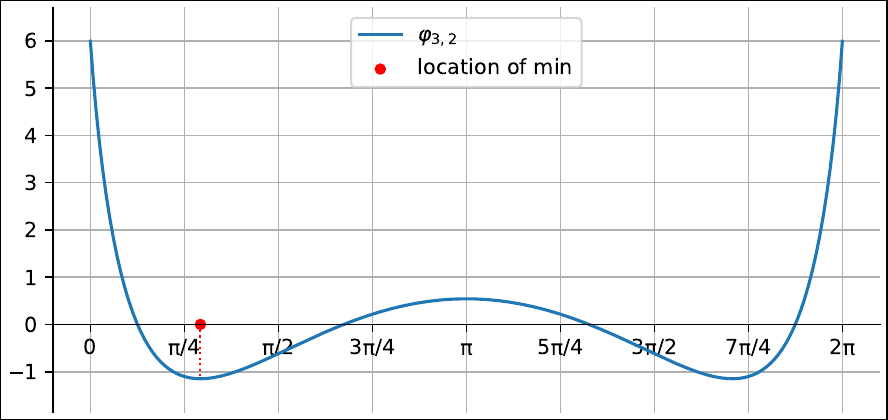} 
\caption{Interaction function $\vphi_{3,2}$ and its minimum}
\label{fig:plot}
\end{figure}

Similar calculations suggest this phenomenon remains true for all $N$ and $\alpha > 2(N-1)$.

\subsection{Optimal placement for 3 points}

\vspace{2mm}

Similar to what was previously done, we prove that the optimal placement of three points is not the equidistribution when $N=2$ and $\alpha = 3$. For that we consider a special, symmetric setup, where $\vth_0 = 0, \vth_1 := \vth \in (0,\pi)$ and $\vth_2 = - \vth \in (-\pi,0)$. With this setup, (5) becomes:

\begin{gather*}
   \Phi(\vth) := 2(2\vphi(\vth)+ \vphi(2\vth)).
\end{gather*}

We will prove that $\Phi(\pi/3) < \Phi(2\pi/3)$. Using the above expression and the fact that $\vphi(2 \pi - \vth) = \vphi(\vth)  $ we have that the above inequality is equivalent to $\vphi(\pi/3)< \vphi(2\pi/3)$. We evaluate: 

\begin{gather*}
    \vphi(\pi/3) = -12 +2\sqrt{3} \pi \approx -1,117,
\end{gather*}

\begin{gather*}
    \vphi(2\pi/3) = 4 \int_0^1 \frac{r-1/2}{1-r+r^2} (3 - 12r + 9r^2)\,dr = -48 + 7\pi\sqrt{3} + 9\log(3) \approx -0.022.
\end{gather*}

and thus:

\begin{gather*}
  \vphi(2\pi/3)   - \vphi(\pi/3) = -36 + 5\pi\sqrt{3} +9\log3 \approx 1.094 > 0.
\end{gather*}

As a result we recover that the optimal placement cannot be when the points are equidistributed, which corresponds to $\vth = \frac{2\pi}{3}$. We may also verify that $\Phi(\pi/6) < \Phi(\pi/3)$. Indeed,

\[
\Phi(\pi/6) < \Phi(\pi/3) \iff 2 \vphi(\pi/6) < \vphi(\pi/3) + \vphi(2\pi/3) .
\]

We calculate:

\begin{align*}
    \vphi(\pi/6) = 4 \int_0^1 \frac{r-\sqrt{3}/2}{1-\sqrt{3}r+r^2} (3 - 12r + 9r^2)\,dr =   \\
    -48 + 18(1 + \sqrt{3}) + (15-12\sqrt{3})\log(2-\sqrt{3})-\frac{5\pi}{2}(3\sqrt{3}-4) \approx -0.6.
\end{align*}

Combining this with the previous calculations shows the desired inequality. This will be useful in section 6.
\vspace{2mm}

\section{Equidistribution arising through additional conditions}

As the previous examples demonstrate, we cannot expect to achieve the optimal distribution by equidistributing the points on the circle. We can do so only if we impose some extra condition on the moments of the points. 

\vspace{2mm}

\textit{Proof of Theorem 1.4}

The argument for the proof follows closely the argument given in \cite{Abakumov2020}. We consider the minimization problem:

\begin{gather*}
    \min_{\vth_0 < \vth_1 < \cdots < \vth_{n-1} \in [0,2\pi)}\sum_{k \neq j}  \vphi_{\alpha_{\ast},N}(\vth_j - \vth_k),
\end{gather*}
which is equivalent to the one stated in Theorem 1.4, according to subsection 3.2. By (4) and (6), we see that the interaction function for $\alpha = \alpha_{\ast}$ is given by:

\begin{gather*}
    \vphi_{\alpha_{\ast},N}(\vth) = \frac{(2N)!}{((N-1)!)^2} \sum_{m=N}^{\infty} \frac{((m-1)!)^2}{(m-N)!(m+N)!} \cos(m \vth).
\end{gather*}

We calculate the second derivative of this series. Here we use that 

\begin{gather*}
\sum_{m = 1}^{\infty}\frac{\sin(m\vth)}{m} = \frac{\pi - \vth }{2} , \quad 0 < \vth < 2\pi.
\end{gather*}

It is known (\cite{Bary2014}, Chapter 1, Section 30) that the partial sums of this series are uniformly bounded. By termwise integration we obtain that:

\begin{gather*}
\sum_{m=1}^{\infty} \frac{1- \cos(m \vth)}{m^2} = \frac{\pi \vth}{2} - \frac{\vth^2}{2},
\end{gather*}
and hence:

\begin{gather*}
\sum_{m=1}^{\infty}\frac{\cos(m \vth)}{m^2} = \frac{\pi^2}{6} - \frac{\pi \vth}{2} + \frac{\vth^2}{2}.
\end{gather*}

Since $ \frac{((m-1)!)^2}{(m-N)!(m+N)!} = \frac{1}{m^2} + O(\frac{1}{m^3})$ as $ m \ra \infty$ we can differentiate the series 

\begin{gather*}
\sum_{m = N}^{\infty}\bigg( \frac{((m-1)!)^2}{(m-N)!(m+N)!} - \frac{1}{m^2} \bigg)\cos(m \vth),
\end{gather*}

and obtain:

\begin{align*}
\frac{((N-1)!)^2}{(2N)!} \cdot  \vphi'_{\alpha_{\ast},N}(\vth) 
&=- \frac{\pi - \vth}{2} + \sum_{m=1}^{N-1}\frac{\sin(m\vth)}{m} +   \sum_{m=N}^{\infty} \bigg( \frac{1}{m} - \frac{m((m-1)!)^2}{(m-N)!(m+N)!}\bigg) \sin(m\vth) .
\end{align*}

Now set:

\begin{gather*}
    a_m = 
\begin{cases}
    1 & m = 0,1,\ldots,N-1, \\
    1- \frac{(m!)^2}{(m-N)!(m+N)!} & m \geq N.
\end{cases}
\end{gather*}

To differentiate the series $\sum_{m = N}^{\infty}\frac{a_m}{m} \sin(m \vth)$, consider an interval $[\eps, 2\pi - \eps]$, $\eps >0$, and the series $\sum_{m = N}^{\infty}a_m \cos(m \vth)$. It can be verified that $a_m\frac{1-r^m}{1-r}$ is eventually decreasing, and has limit zero. Moreover the partial sums \{ $\sum_{m=N}^{M}\cos(m \vth)\}_{M \geq N}$ are uniformly bounded on the interval $[\eps, 2\pi - \eps]$. Therefore, by Dirichlet's summation formula we conclude that for $0 \leq r < 1$: 

\begin{gather*}
    \bigg| \sum_{m=N}^{\infty}a_m (1-r^m)\cos(m \vth) \bigg| \leq C_{\eps} (1-r), \,\, \vth \in [\eps, 2\pi - \eps].
\end{gather*}

As a result, for every $[x,y] \sbs [\eps,2\pi - \eps]$ we have that:

\begin{gather*}
    \bigg| \int_x^y \sum_{m=N}^{\infty}a_m\cos(m \vth) \, d\vth - \int_x^y \sum_{m=N}^{\infty}a_m r^m \cos(m \vth) \, d\vth\bigg| \leq 2\pi C_{\eps} (1-r).
\end{gather*}

Since 

\begin{gather*}
\lim_{r \ra 1} \int_x^y  \sum_{m=N}^{\infty}a_m r^m \cos(m \vth) \, d\vth = \sum_{m=N}^{\infty}\frac{a_m}{m} \sin(m x) -  \sum_{m=N}^{\infty}\frac{a_m}{m} \sin(m y),
\end{gather*}
we conclude, by the fundamental theorem of calculus, that:

\begin{gather*}
\frac{d}{d\vth} \bigg( \sum_{m=N}^{\infty}\frac{a_m}{m} \sin(m \vth)\bigg) = \sum_{m = N}^{\infty} a_m \cos(m \vth).
\end{gather*}

Overall,

\begin{gather*}
   \frac{((N-1)!)^2}{(2N)!} \cdot  \vphi_{\alpha_{\ast},N} ''(\vth) = \frac{1}{2} + \sum_{m=1}^{N-1} \cos(m \vth) + \sum_{m=N}^{\infty} \bigg( 1- \frac{(m!)^2}{(m-N)!(m+N)!}\bigg)\cos(m\vth).
\end{gather*}

At this point we would like to make use of Bari's theorem (Theorem 4.1), but unfortunately the coefficients $\{ a_m \}_{m \geq 0}$ are not convex. Instead it can be easily verified that they satisfy the conditions of Proposition 4.3. Moreover it can be verified that for $m \geq N$, $(m+1)a_m - m a_{m+1} \leq 1 = a_0$ if and only if $m \geq N^2-1$ with equality occurring for $m=N^2-1$. This means that we may decrease the coefficients $a_1,a_2,\ldots, a_{N^2-2}$ of the series to obtain a positive cosine series, according to Bari's theorem. Let $\tilde{a}_m$ represent the modified coefficients. Then the function:

\begin{gather*}
\psi(\vth) : = \frac{((N-1)!)^2}{(2N)!} \cdot  \vphi_{\alpha_{\ast},N} ''(\vth) - \sum_{m=1}^{N^2-2}(a_m - \tilde{a}_m)\cos(m\vth),
\end{gather*}

is positive, and hence $\vphi_{\alpha_{\ast},N}''$ can be written as:

\begin{gather*}
    \vphi_{\alpha_{\ast},N}''(\vth) =  \frac{(2N)!}{((N-1)!)^2}\bigg( \psi(\vth) + \sum_{m=1}^{N^2-2}(a_m - \tilde{a}_m)\cos(m\vth) \bigg).
\end{gather*}

We note here that the function $\psi(\vth) $ might have some zeros, but they are countable as mentioned in Remark 2.2. As a result we have:

\begin{gather}
    \vphi_{\alpha_{\ast},N}(\vth) = \frac{(2N)!}{((N-1)!)^2}\bigg( \tilde{\vphi}(\vth) - \sum_{m=1}^{N^2-2}\frac{(a_m - \tilde{a}_m)}{m^2}\cos(m\vth) \bigg).
\end{gather}
where $\tilde{\vphi}$ satisfies $\tilde{\vphi}'' = \psi$ and is strictly convex. We may now proceed to prove the equivalence stated in the theorem. On one hand, we always have:

\begin{gather*}
         \min_{(\vth_1,\vth_2,\ldots,\vth_{n-1}) \in W}\sum_{k \neq j}  \vphi_{\alpha_{\ast},N}(\vth_j - \vth_k) \leq \sum_{k \neq j} \vphi_{\alpha_{\ast},N} \bigg( \frac{2j\pi}{n} - \frac{2k\pi}{n}\bigg).
\end{gather*}

Consider now $\vth_0 < \vth_1 < \cdots < \vth_{n-1} \in [0,2\pi)$ such that $\sum_{k=0}^{n-1}e^{im\vth_k} = 0 $ for all $m=1,2,\ldots,N^2-2$. By equality (14) we may write:

\begin{gather*}
    \sum_{j \neq k} \vphi_{\alpha_{\ast},N} (\vth_j - \vth_k) = \frac{(2N)!}{((N-1)!)^2}\bigg(\sum_{j \neq k} \tilde{\vphi}(\vth_j - \vth_k)  - \sum_{m=1}^{N^2-2}\frac{(a_m - \tilde{a}_m)}{m^2}\sum_{j \neq k}\cos(m(\vth_j-\vth_k))\bigg).
\end{gather*}

 A simple calculation shows that:

\begin{gather}
   \sum_{j \neq k} \cos(m(\vth_j - \vth_k)) = \bigg| \sum_{k=0}^{n-1}e^{im\vth_k} \bigg|^2 - n = \sum_{j \neq k} \cos \bigg(\frac{2j\pi}{m} - \frac{2k\pi}{m} \bigg), \quad 1 \leq m \leq N^2-2.
\end{gather}

Therefore:

\begin{gather*}
    \sum_{j \neq k} \vphi_{\alpha_{\ast},N}(\vth_j - \vth_k) = \\
\frac{(2N)!}{((N-1)!)^2}\bigg(\sum_{j \neq k}\tilde{\vphi}(\vth_j - \vth_k)- \sum_{m=1}^{N^2-2}\frac{(a_m - \tilde{a}_m)}{m^2}\sum_{j \neq k}\cos\bigg(\frac{2j\pi}{n} - \frac{2k\pi}{n} \bigg)\bigg).
\end{gather*}

At this point we will apply the following result from \cite{Abakumov2020}.

\begin{manuallemma}
\textit{Let $\vphi$ be a $2\pi$-periodic even continuous function strictly convex on $(0,2\pi)$. Then for any $n \geq 2$ we have}

\begin{gather*}
\inf_{\vth_j \in [0,2\pi), 0 \leq j < n} \sum_{0 \leq j,k <n, j \neq k} \vphi (\vth_j - \vth_k) = \sum_{0 \leq j,k <n, j \neq k}\vphi\bigg( \frac{2\pi j}{n} -\frac{2\pi k}{n} \bigg).
\end{gather*}
\textit{Furthermore, if $\vth_j \in [0,2\pi), 0 \leq j < n$, and}

\begin{gather*}
\sum_{0 \leq j,k <n, j \neq k} \vphi (\vth_j - \vth_k) = \sum_{0 \leq j,k <n, j \neq k}\vphi\bigg( \frac{2\pi j}{n} -\frac{2\pi k}{n} \bigg),
\end{gather*}
\textit{then the points $e^{i \vth_j}$ are equispaced on the unit circle.}

\end{manuallemma}

Since the function $\tilde{\vphi}$ is strictly convex, we obtain:

\begin{gather*}
\frac{(2N)!}{((N-1)!)^2} \sum_{j \neq k}\tilde{\vphi}(\vth_j - \vth_k) \geq \frac{(2N)!}{((N-1)!)^2}\sum_{j \neq k}\tilde{\vphi}\bigg(\frac{2j\pi}{n} - \frac{2k\pi}{n} \bigg),
\end{gather*}

with equality holding if and only if the points $\vth_k$ are equidistributed. So,

\begin{gather*}
\sum_{j \neq k} \vphi_{\alpha_{\ast},N}(\vth_j - \vth_k) \geq \\ 
\frac{(2N)!}{((N-1)!)^2}\bigg( \sum_{j \neq k}\tilde{\vphi}\bigg(\frac{2j\pi}{n} - \frac{2k\pi}{n} \bigg)- \sum_{m=1}^{N^2-2}\frac{(a_m - \tilde{a}_m)}{m^2}\sum_{j \neq k}\cos\bigg(\frac{2j\pi}{n} - \frac{2k\pi}{n} \bigg) \bigg) \\
= \sum_{j \neq k}\vphi_{\alpha_{\ast},N}\bigg(\frac{2j\pi}{n} - \frac{2k\pi}{n} \bigg).
\end{gather*}

This concludes the proof of the theorem. $\hfill \blacksquare$

\begin{remark}
This technique does not work for $\alpha \neq \alpha_{\ast}$. When $\alpha> \alpha_{\ast}$ the coefficients $a_m$ do not converge to zero. When $\alpha< \alpha_{\ast}$ the coefficients $a_m$ are not always positive and tend to $-\infty$. 
\end{remark}

\begin{remark}
The hypothesis that the first $N^2-2$ moments of the points $\vth_0, \vth_1, \ldots, \vth_{n-1}$ be zero seems quite restrictive, but the calculations demonstrate that for a low number of points some restrictions are necessary. Indeed, as was calculated in section 5, for two points the equidistribution is not minimal. Moreover, the configuration $\vth_0 = 0, \vth_1 = \pi/3$ is of lower energy than the equidistribution and none of the first two moments is zero.

Similarly, for three points, equidistribution is seen to be non-optimal, and the configuration $\vth_0 = 0, \vth_1 = \pi/6, \vth_2 = -\pi/6$ is of lower energy than equidistribution and none of its first two moments is zero.
\end{remark} 

\section{Asymptotics for the minimum}

In this section we prove a sharp asymptotic estimate for the norm of the equidistributed configuration. We are able to calculate the exact value for the critical exponent. We then proceed to prove the main theorem of this section which states that asymptotically, equidistribution is optimal, up to a constant.  

\vspace{2mm}

\it{Proof of Theorem 1.5:}

Consider the equidistribution function, $\Psi_n^N(z) = \sum_{k=0}^{n-1} \frac{1}{(z-e^{2  k\pi i /n })^N}$. Consider also the polynomial $p(z) = z^n -1$. $\Psi_n^N$ may be rewritten in terms of $p$ as follows:

\begin{gather*}
    \Psi_n^N(z) = \frac{(-1)^N}{(N-1)!} \bigg( \frac{p'(z)}{p(z)}\bigg)^{(N-1)}.
\end{gather*}

As a consequence we obtain the following power series representation:

\begin{gather*}
    \Psi_n^N(z) = \frac{(-1)^{N+1}}{(N-1)!} \sum_{k=0}^{\infty} n (n(k+1)-1)_{N-1}z^{n(k+1)-N}.
\end{gather*}

Let $c_{n,k}^N = \frac{(-1)^{N+1}}{(N-1)!}  n (n(k+1)-1)_{N-1}$. For $\alpha > 2(N-1)$ we have:

\begin{gather*}
    \Vert \Psi_n^N \Vert_{\alpha}^2 = \int_{\D}|\Psi_n^N(z)|^2 \, dA_{\alpha}(z) = \sum_{k=0}^{\infty} |c_{n,k}^N|^2 \int_{\D}|z|^{2n(k+1) - 2N} \,dA_{\alpha}(z) .
\end{gather*}

A simple calculation show that the values of the integrals of these monomials are equal to $\frac{\Gamma(n(k+1)-N+1) \Gamma(\alpha+2)}{\Gamma(n(k+1)-N + \alpha + 2 )}$. We have:

\begin{gather*}
    |c_{n,k}^N|^2  = \frac{n^{2N}}{((N-1)!)^2} \prod_{j=1}^{N-1}(k+1 - j/n)^2,
\end{gather*}    

and,

\begin{gather*}
    \frac{\Gamma(n(k+1)-N+1) \Gamma(\alpha+2)}{\Gamma(n(k+1)-N + \alpha + 2 )} = \frac{\Gamma(\alpha + 2)}{n^{\alpha + 1} (k+1 - N/n)^{\alpha + 1}} \cdot  \frac{(n(k+1)-N)^{\alpha + 1} \Gamma(n(k+1)-N+1)}{\Gamma(n(k+1)-N + \alpha + 2)}.
\end{gather*}

Hence the norm of $\Psi_n^N$ can we rewritten as the series:

\begin{gather*}
     \Vert \Psi_n^N \Vert_{\alpha}^2 =  \\ 
     n^{2N-\alpha -1}\cdot \frac{\Gamma(\alpha+2)}{((N-1)!)^2} \sum_{k=0}^{\infty}\bigg( \frac{(n(k+1)-N)^{\alpha + 1} \Gamma(n(k+1)-N+1)}{(k+1-N/n)^{\alpha+1}\Gamma(n(k+1)-N + \alpha + 2)}\prod_{j=1}^{N-1}(k+1 - j/n)^2  \bigg) .
\end{gather*}

By standard properties of the Gamma function we have that:

\begin{gather*}
    \frac{(n(k+1)-N)^{\alpha + 1} \Gamma(n(k+1)-N+1)}{\Gamma(n(k+1)-N + \alpha + 2)} \ra 1,
\end{gather*}
as $n \ra \infty$. Hence the coefficients of the above series can be easily seen to converge to $(k+1)^{2(N-1)- \alpha -1}$ as $n \ra \infty$. Since $\alpha > 2(N-1)$ the series $\sum_{k=0}^{\infty} (k+1)^{2(N-1)- \alpha -1}$ converges, and we may conclude that:

\begin{gather*}
  \lim_{n \ra \infty} n^{\alpha + 1-2N}  \Vert \Psi_n^N \Vert_{\alpha}^2 = \frac{\Gamma(\alpha +2) \zeta(\alpha + 1 - 2(N-1))}{((N-1)!)^2}.
\end{gather*}

$\hfill \blacksquare$

\vspace{2mm}

\it{Proof of Theorem 1.6}

One direction of the inequality is obvious. To prove the other direction it is sufficient to prove that for fixed N, there exists a constant $C> 0$, independent of $n \in \N$ such that for any $a_k \in \T$, $0 \leq k \leq N-1$ we have:

\begin{gather*}
 \bigg \Vert \sum_{k=0}^{n-1} \frac{1}{(z-a_k)^N} \bigg \Vert_{\alpha}^2 \geq  \frac{C}{n^{\alpha+1 -2N}}.
\end{gather*}

To that end,  we write:

\begin{gather*}
    \bigg \Vert \sum_{k=0}^{n-1} \frac{1}{(z-a_k)^N} \bigg \Vert_{\alpha}^2 = \int_{\D} \bigg|\sum_{k=0}^{n-1} \frac{1}{(z-a_k)^N}  \bigg|^2 \, dA_{\alpha}(z) \\
    = \frac{1}{((N-1)!)^2} \sum_{s = N-1}^{\infty} ((s)_{N-1})^2 \bigg| \sum_{k=0}^{n-1} a_k^{s+1}\bigg|^2 \no{z^{s-N+1}}_{\alpha}^2 \\
    \geq \frac{1}{((N-1)!)^2} \sum_{s = N-1}^{2nN-1} ((s)_{N-1})^2 \bigg| \sum_{k=0}^{n-1} a_k^{s+1}\bigg|^2 \no{z^{s-N+1}}_{\alpha}^2.
\end{gather*}

It is obvious that $((s)_{N-1})^2 \gtrsim s^{2(N-1)}$, for $s \geq N-1$. Furthermore, $\no{z^{s-N+1}}_{\alpha}^2 = \frac{\Gamma(\alpha+2)\Gamma(s-N+2)}{\Gamma(s-N+ \alpha + 3)} \asymp s^{-\alpha - 1}$, for $s \geq N-1$, by Stirling's formula. Using the fact that $s = O(n)$ we deduce that:

\begin{gather*}
\bigg \Vert \sum_{k=0}^{n-1} \frac{1}{(z-a_k)^N} \bigg \Vert_{\alpha}^2 \gtrsim \sum_{s = N-1}^{2nN-1} \frac{1}{s^{ \alpha +3 -2N}} \bigg| \sum_{k=0}^{n-1} a_k^{s+1}\bigg|^2 \gtrsim \frac{1}{n^{\alpha +3 - 2N}}  \sum_{s = N-1}^{2nN-1} \bigg| \sum_{k=0}^{n-1} a_k^{s+1}\bigg|^2.
\end{gather*}
Therefore it suffices to show that:

\begin{gather*}
    \sum_{s=N-1}^{2nN-1} \bigg| \sum_{k=0}^{n-1} a_k^{s+1}\bigg|^2 \geq C n^2,
\end{gather*}
for some constant $C>0$ and all unimodular numbers $a_k \in \T$ and all $n \in \N$. We have:

\begin{gather*}
    \sum_{s=N-1}^{2nN-1} \bigg| \sum_{k=0}^{n-1} a_k^{s+1}\bigg|^2 = \sum_{s=N}^{2nN} \bigg| \sum_{k=0}^{n-1} a_k^s\bigg|^2 = \sum_{s=1}^{2nN} \bigg| \sum_{k=0}^{n-1} a_k^s\bigg|^2  - \sum_{s=1}^{N-1} \bigg| \sum_{k=0}^{n-1} a_k^s\bigg|^2 \\
    \geq  \sum_{s=1}^{2nN} \bigg| \sum_{k=0}^{n-1} a_k^s\bigg|^2 - n^2 \cdot (N-1).
\end{gather*}

In \cite{Abakumov2020}, section 4, it is proven that for every $M \in \N$ the following inequality holds for all $a_k \in \T$ and all $n \in \N$:

\begin{gather*}
    \sum_{s=1}^{M} \bigg| \sum_{k=0}^{n-1} a_k^s\bigg|^2 \geq \frac{n(M - n + 1)}{2}. 
\end{gather*}

Our inequality follows with $C = 1/2$ by applying the above inequality for $M = 2nN$. $\hfill \blacksquare$ 

\section{Closure of simple partial fractions}

In this final section, the closure of the set $SF^N(\T)$ in $A^2_{\alpha}(\D)$ is studied. We begin by establishing some propositions that will be combined into Theorems 1.7. Given some integer $N \in \N$ we consider $\alpha_{\ast} = \alpha_{\ast}(N) = 2(N-1) + 1 = 2N-1$. 

\begin{proposition}
Let $N \in \N$. Then $SF^N$ is nowhere dense in $A^2_{\alpha}$ when $2(N-1)< \alpha < \alpha_{\ast}$. Additionally, $SF^N$ is not dense in $A^2_{\alpha}$ when $a = \alpha_{\ast}$.
\end{proposition}

\begin{proof}
We first consider the case when $2(N-1) < \alpha < \alpha_{\ast}$, and some function $f \in A^2_{\alpha} \backslash SF^N$. If $f$ belongs to the closure of the partial fractions in $A^2_{\alpha}$, then there is a sequence $f_n \in SF^N$, 

\begin{gather*}
f_n(z) = \sum_{0 \leq k < N_n} \frac{1}{(z-a_{n,k})^N},
\end{gather*}
which tends to $f$ in norm. The degree $N_n$ must tend to infinity, otherwise a suitable subsequence converges to a partial fraction, and by uniqueness of limit, $f$ must be a partial fraction itself. Next, we have $ \no{f_n}_{\alpha} \leq \no{f_n-f}_{\alpha} + \no{f}_{\alpha} < \no{f}_{\alpha} + 1$ for $n$ large enough. By Theorems 1.5 and 1.6 we have:

\begin{gather*}
\no{f_n}_{\alpha}^2 \geq \min_{a_{n,k} \in \T} \bigg \Vert \sum_{0 \leq k < N_n} \frac{1}{(z- a_{n,k})^N} \bigg \Vert_{\alpha}^2 \gtrsim \no{\Psi_{N_n}^N}_{\alpha}^2 \gtrsim N_n^{2(N-1)+1 - \alpha}.
\end{gather*}

Since $\alpha < \alpha_{\ast}$ we obtain that $\no{f_n}_{\alpha}$ tends to infinity, which gives a contradiction. The case $\alpha = \alpha_{\ast}$ is similar. Suppose that $SF^N$ is dense in $A^2_{\alpha_{}\ast}$. Then we would be able to approximate the function $f(z) = 0$. We would thus have 

\begin{gather*}
      \bigg \Vert  \sum_{0 \leq k < N_n} \frac{1}{(z-a_{n,k})^N}\bigg \Vert_{\alpha_{\ast}} \ra 0,
\end{gather*}
for some sequence of partial fractions. Moreover, by Theorem 1.6, we know that:

\begin{gather*}
\bigg \Vert \sum_{0 \leq k < N_n} \frac{1}{(z-a_{n,k})^N} \bigg \Vert_{\alpha_{\ast}}^2 \gtrsim \Vert \Psi_n^N \Vert_{\alpha_{\ast}}.
\end{gather*}

By Theorem 1.5 we have that $ \Vert \Psi_n^N \Vert_{\alpha_{\ast} } \geq C$, for some constant $C>0$ independent of $n$. This gives a contradiction.
\end{proof}

\begin{lemma}
Let $N \in \N$. If $\alpha > \alpha_{\ast}$ then $\overline{SF^N}$ is an additive subgroup of $A^2_{\alpha}$.
\end{lemma}

\begin{proof}
It is clear that $SF^N$ is an additive semi-group. Since $\alpha > \alpha_{\ast}$, Theorem 1.5 guarantees that $\Vert \Psi_n^N \Vert_{\alpha} \ra 0$, hence $0 \in \overline{SF^N}$. It suffices to show that for a given $a \in \T$, the function $- \frac{1}{(z-a)^N}$ belongs to  $\overline{SF^N}$. Since the roots of unity are dense in $\T$ we can find positive integer sequences $k_m, n_m$ such that $ \exp{\bigg( 2\pi i \frac{k_m}{n_m} \bigg)} \ra a$. Let $\phi_m =  2\pi \frac{k_m}{n_m}$ and consider the functions:

\begin{gather*}
    f_m(z) = \frac{1}{(z-e^{i\phi_m})^N}.
\end{gather*}

The function $f_m(z)$ converges to $\frac{1}{(z-a)^N}$ uniformly on compact sets of $\D$. Next,

\begin{gather*}
    \int_{s < |z| <1 }\bigg| |f_m(z)|^2 (1-|z|^2)^{\alpha} \,dA(z) \lesssim \int_s^1 r^{\alpha - 2N} \,dr,
\end{gather*}
uniformly in $m,s$. Hence $f_m(z) \ra \frac{1}{(z-a)^N}$ in $A^2_{\alpha}$. Now consider the functions $\Psi_{n_m}^N$ and write:

\begin{gather*}
    \bigg\Vert - \frac{1}{(z-a)^N} - \bigg(\Psi_{n_m}^N - f_m \bigg)  \bigg\Vert_{\alpha} \\
      = \bigg\Vert - \bigg( \frac{1}{(z-a)^N}- f_m\bigg) - \Psi_{n_m}^N  \bigg\Vert_{\alpha} \\
       \leq \bigg\Vert  \frac{1}{(z-a)^N}- f_m \bigg\Vert_{\alpha} + \bigg\Vert  \Psi_{n_m}^N  \bigg\Vert_{\alpha},
\end{gather*}
which converges to zero. Since $\Psi_{n_m}^N - f_m \in SF^N$ we obtain $- \frac{1}{(z-a)^N} \in \overline{SF^N}$.
\end{proof}

\begin{proposition}
Let $N \in \N$. Then $SF^N$ is dense in $A^2_{\alpha}$ when $\alpha > \alpha_{\ast}+1$.
\end{proposition}

\begin{proof}
We will use Theorem 5 in \cite{Borodin2014} to prove that the set $SF^N$ contains the real vector space spanned by $\bigg\{ \frac{1}{(z-a)^N} \, : \, a \in \T \bigg\}$. By Lemma 8.2, the closure of partial fractions is an additive subgroup of $A^2_{\alpha}$. The Bergman space being a Hilbert space is uniformly smooth with modulus of smoothness $s(\tau) = O(\tau^2)$. We consider the map $$\iota : \T \ra A^2_{\alpha},$$ $$\iota(a) =  \frac{1}{(z-a)^N}.$$ We have that $\frac{d \iota}{d a} (a) = -\frac{N}{(z-a)^{N+1}}$. Since $\alpha > \alpha_{\ast} + 1 = 2N$, by Proposition 3.3 we get that $\no{\frac{d \iota}{d a}}_{\alpha} <  \infty$. This norm is uniformly bounded in $a \in \T$ by radial symmetry, from which we conclude that the map $\iota$ is Lipschitz. The hypotheses of Theorem 5 in \cite{Borodin2014} are satisfied, and hence $SF^N$ contains a real vector space $M$ spanned by $\bigg\{ \frac{1}{(z-a)^N} \, : \, a \in \T \bigg\}$. We will show that $M$ is dense in $A^2_{\alpha}$. Consider some function $f \in A^2_{\alpha}$ that is orthogonal to $M$, i.e:

\begin{gather}
    \int_{\D} \frac{\overline{f(z)}}{(z-a)^N} \,dA_{\alpha}(z) = 0 \, ,\quad \forall a \in \T.
\end{gather}

We consider the function $H: \C \setminus \overline{\D}  \ra \C$:

\begin{gather*}
H(w) = \int_{\D} \frac{\overline{f(z)}}{(z-w)^N} \,dA_{\alpha}(z),
\end{gather*}
which is analytic in $\C \setminus \overline{\D}$ and conntinuous up to $\C \setminus \D$. Since $H \big\vert_{\partial \D} = 0$, by the uniqueness theorem we obtain that $H \big\vert_{\C \setminus \overline{\D}} = 0$. Since 

\begin{gather*}
\lim_{w \ra \infty} \int_{\D} \frac{\overline{f(z)}}{(z-w)^k} \,dA_{\alpha}(z) = 0, \,\, k \geq 1,
\end{gather*}
we obtain that:

\begin{gather*}
\int_{\D} \frac{\overline{f(z)}}{(z-w)^k} \,dA_{\alpha}(z) = 0, \,\, k \geq 1 , \,\, w \in \C \setminus \D.
\end{gather*}
Hence,
\begin{gather*}
     \int_{\D} \overline{f(z)} R(\frac{1}{z-a}) \,dA_{\alpha}(z) = 0, 
\end{gather*}
for all polynomials $R$ with $R(0)=0$ and $a \in \C \backslash \D$. Since these polynomials are dense in $A^2_{\alpha}$, we get that $f = 0$. As a result $M = A^2_{\alpha}$.
\end{proof}

Theorem 1.7 follows as a Corollary of Proposition 8.1 and Proposition 8.3.


\vspace{5mm}

\begin{scriptsize}
\noindent \bf{ACKNOWLEDGEMENTS.} 
\end{scriptsize}
I would like to thank my advisors Evgeny Abakumov and Alexander Borichev for proposing me the problems discussed in this article, and for numerous discussions and suggestions.

%
%

\printbibliography[heading=bibintoc,title={References}]

%
%

\vspace{8mm}

NIKIFOROS BIEHLER: 

UNIV GUSTAVE EIFFEL,

UNIV PARIS EST CRETEIL, 

CNRS, LAMA UMR8050 

F-77447 MARNE-LA-VALÉE, FRANCE
 
 Email address : \bf{nikiforos.biehler@univ-eiffel.fr}

\end{document}